\documentclass{amsart}

\usepackage{float}
\usepackage[dvipdfmx]{graphicx}

\newtheorem{theo}{Theorem}[section]
\newtheorem{prop}[theo]{Proposition}
\newtheorem{lem}[theo]{Lemma}
\newtheorem{cor}[theo]{Corollary}

\theoremstyle{definition}
\newtheorem{defi}{Definition}[section]

\theoremstyle{remark}
\newtheorem*{rem}{Remark}

\begin{document}

\title[On handlebody-knot pairs]{On handlebody-knot pairs which realize exteriors of knotted surfaces in $S^3$}
\author{Shundai Osada}
\address{Faculty of Mathematics, Kyushu University, 744, Motooka, Nishi-ku, Fukuoka, 819-0395, Japan}
\email{s.osada@kyudai.jp}

\begin{abstract}
	In this paper, we describe the relation between the study of closed connected surfaces embedded in $S^3$ and the theory of handlebody-knots. By Fox's theorem, a pair of handlebody-knots is associated to a closed connected surface embedded in $S^3$ in the sense that their exterior components are pairwise homeomorphic. We show that for every handlebody-knot pair associated to a genus two ``prime bi-knotted'' surface, one is irreducible, and the other is reducible. Furthermore, for given two genus two handlebody-knots $H_1$ and $H_2$ satisfying certain conditions, we will construct a ``prime bi-knotted'' surface whose associated handlebody-knot pair coincides with $H_1$ and $H_2$.

\end{abstract}

\maketitle

\section{Introduction}

The study of closed connected surfaces embedded in $S^3$ has been investigated since around 1970. Waldhausen \cite{Waldhausen} showed that for any two Heegaard surfaces of $S^3$ which have the same genus, there exists an isotopy of $S^3$ taking one onto the other. Suzuki \cite{Suzuki} showed that for any positive integer $g$, there exists a prime genus $g$ surface embedded in $S^3$ (see Section~\ref{section:Surface in S^3} for the explicit definition of a prime surface). Tsukui \cite{Tsukui-1, Tsukui-2} gave a necessary and sufficient condition for the primeness of a closed connected surface embedded in $S^3$ for the case of genus two. 

A handlebody-knot is a handlebody embedded in $S^3$, and the theory of such objects has been extensively studied and developed recently (for example, see \cite{Ishii, I-K-M-S}). According to a classical theorem of Fox \cite{Fox}, for every closed connected surface embedded in $S^3$, there exist two handlebody-knots whose exteriors are homeomorphic to the two components of the exterior of the surface, respectively. Thus, a pair of handlebody-knots is associated to such a surface, although they are not unique in general (see \cite{Lee}, for example). On the other hand, there exists a closed connected surface which does not bound handlebodies on either side. We say that such a surface is ``bi-knotted". 

The purpose of this paper is to describe the relation between the study of surfaces embedded in $S^3$ and the theory of handlebody-knots. We first show that for every handlebody-knot pair associated to a genus two prime bi-knotted surface, one is irreducible and the other is reducible (see Proposition~\ref{prop:irr-red}). Furthermore, given two genus two handlebody-knots $H_1$ and $H_2$ satisfying certain conditions, we will construct a prime bi-knotted surface whose associated handlebody-knot pair coincides with $H_1$ and $H_2$. 

The paper is organized as follows. In Section~\ref{section:Surface in S^3}, we give some definitions and notation about closed connected surfaces embedded in $S^3$. In Section~\ref{section:handlebody-knot}, we give fundamental definitions for handlebody-knots. In Section~\ref{section:handlebody-knot-pair}, we study pairs of handlebody-knots associated to closed connected surfaces embedded in $S^3$. Finally, in Section~\ref{section:main theorem}, we state and prove our main theorem (Theorem~\ref{theo:main}).

\section{Surfaces in $S^3$}\label{section:Surface in S^3}

Throughout this paper, we work in the PL category, and a surface in $S^3$ means a closed connected orientable surface embedded in $S^3$.
In this section, we give some definitions and notation about surfaces in $S^3$. 
\begin{defi}
	Let $F$ be a surface in $S^3$. The \emph{exterior} of $F$ is the closure of $S^3 \setminus N(F)$ in $S^3$, where $N(F)$ denotes the regular neighborhood of $F$ in $S^3$. By virtue of the Alexander duality, the exterior has exactly two connected components, which we denote by $V_F$ and $W_F$.
	\begin{itemize}
		\item[(1)]We say that $F$ is \emph{unknotted} if both $V_F$ and $W_F$ are homeomorphic to handlebodies.
		\item[(2)]We say that $F$ is \emph{knotted} if it is not unknotted, i.e.\ if at least one of  $V_F$ and $W_F$ is not homeomorphic to a handlebody.
		\item[(3)]We say that $F$ is \emph{bi-knotted} if neither $V_F$ nor $W_F$ is homeomorphic to a handlebody. 
	\end{itemize}
\end{defi}

\begin{figure}[ht]
		\includegraphics[width=12cm]{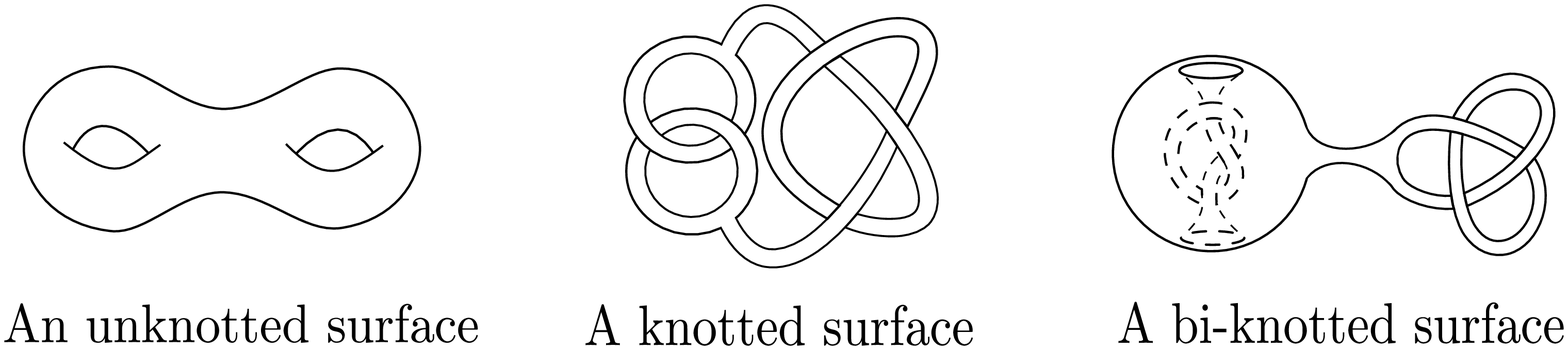}
		\caption{Surfaces of genus two in $S^3$}
		\label{surfaces}
\end{figure}

See Figure~\ref{surfaces} for explicit examples.

\begin{rem}
	Surfaces of genus zero are always unknotted, and the genus of a bi-knotted surface is always greater than or equal to two. These are direct consequences of Alexander's results \cite{Alexander}.
\end{rem}

Two surfaces in $S^3$ are said to be \emph{equivalent} if there exists an isotopy of $S^3$ taking one onto the other. Waldhausen \cite{Waldhausen} showed that two unknotted surfaces which have the same genus are equivalent. Furthermore, we see easily that the study of equivalence classes of surfaces that are not bi-knotted is reduced to the theory of handlebody-knots (see Section~\ref{section:handlebody-knot} of the present paper). Such surfaces appear as boundaries of handlebody-knots. Therefore, in this paper, we will focus on the study of bi-knotted surfaces.

\begin{defi}
	Let $F_1$ and $F_2$ be surfaces in $S^3$. Suppose that there exists a 3-ball $B^3$ in $S^3$ such that $F_1 \subset {\rm int}\, B^3$ and $F_2 \subset S^3\setminus B^3$. Note that $S^3\setminus (F_1 \cup F_2)$ consists of exactly three connected components, and $\partial B^3$ is contained in one of the components. Let $V$ denote the closure of the component of $S^3\setminus (F_1 \cup F_2)$ which contains $\partial B^3$.  Let $h:D^2\times [-1,1]\to V$ be an embedding such that 
	\begin{eqnarray*}
	 	h(D^2\times [-1,1])\cap F_1 & = & h(D^2\times \{ -1\} ),\\	 	
	 	h(D^2\times [-1,1])\cap F_2 & = & h(D^2\times \{ 1\} ),\, {\rm and}\\
	 	h(D^2\times [-1,1])\cap \partial B^3 & = & h(D^2\times \{ 0\} ).
	 \end{eqnarray*}
Then, we say that the surface $F_1 \sharp F_2$ defined by
	\begin{eqnarray*}
		F_1\sharp F_2=F_1\cup F_2 \cup h(\partial D^2 \times [-1,1])\setminus h({\rm int}\,D^2\times \{ 1,-1\} )
	\end{eqnarray*}
is the \emph{isotopy sum} of $F_1$ and $F_2$. It is known that this does not depend on a particular choice of the embedding $h$ up to isotopy (for details, see \cite{Tsukui-1}).
\end{defi}

See Figure~\ref{isotopy-sum} for an explicit example. 

\begin{figure}[ht]
		\includegraphics[width=12cm]{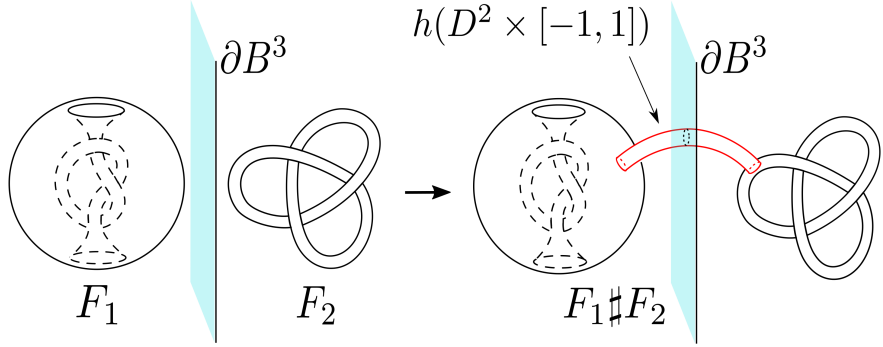}
		\caption{The isotopy sum of $F_1$ and $F_2$}
		\label{isotopy-sum}

\end{figure}

\begin{defi}
	Let $F$ be a surface in $S^3$. We say that $F$ is \emph{prime} if for an arbitrary isotopy sum $F_1 \sharp F_2$ of surfaces $F_1$ and $F_2$ in $S^3$ that is equivalent to $F$, either $F_1$ or $F_2$ is equivalent to $S^2$ in $S^3$.
\end{defi}

For example, the bi-knotted surface in Figure~\ref{surfaces} is not prime, since it is equivalent to the isotopy sum of two knotted tori. The study of genus two non-prime surfaces in $S^3$ is reduced to the standard knot theory, since the factors of such surfaces are tori in $S^3$ and they appear as boundaries of regular neighborhoods of classical knots in $S^3$ by virtue of Alexander's torus theorem \cite{Alexander}. Hence, in this paper, we will focus on prime bi-knotted surfaces. Homma's example (see Figure~\ref{homma-sf}) is an example of a prime bi-knotted surface of genus two.

\begin{figure}[ht]
        \includegraphics[width=8cm]{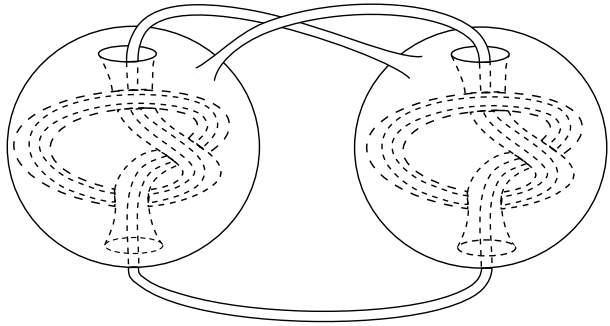}
        \caption{Homma's example \cite{Homma}}
        \label{homma-sf}
\end{figure}

\section{Handlebody-knots}\label{section:handlebody-knot}

In this section, we give fundamental definitions for handlebody-knots. A \emph{handlebody-knot} is a handlebody embedded in $S^3$. Two handlebody-knots are \emph{equivalent} if there exists an isotopy of $S^3$ taking one onto the other. A handlebody-knot is \emph{trivial} if it is equivalent to a handlebody standardly embedded in $S^3$. In other words, a handlebody-knot is trivial if its exterior is homeomorphic to a handlebody. A handlebody-knot is \emph{non-trivial} if it is not equivalent to a trivial handlebody-knot. For explicit examples, see Figure~\ref{handlebody-knots}. Throughout the paper, we adopt the drawing convention that handlebodies are depicted in gray, while surfaces are depicted transparently.

\begin{figure}[ht]
        \includegraphics[width=12cm]{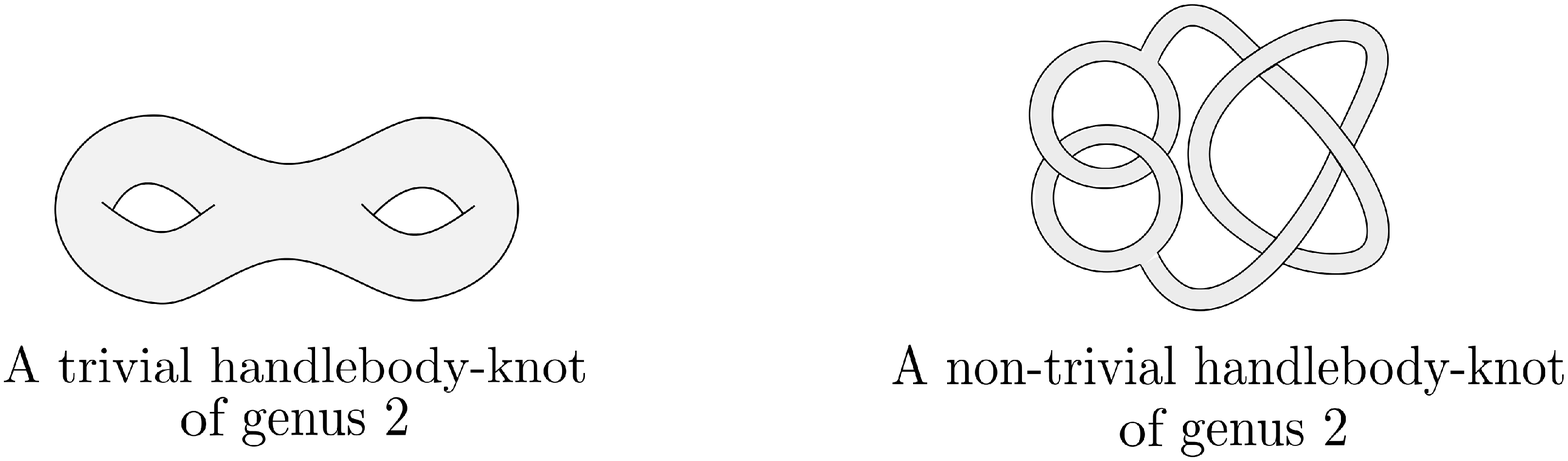}
        \caption{Handlebody-knots of genus two}
        \label{handlebody-knots}
\end{figure}

A \emph{spatial trivalent graph} is a finite trivalent graph embedded in $S^3$. There exist good relations between handlebody-knots and spatial trivalent graphs. For example, the regular neighborhoods of theta-graphs and handcuff graphs (see Figure~\ref{graphs}) embedded in $S^3$ are genus two handlebody-knots, and for every genus two handlebody-knot, there exists a theta-graph or a handcuff graph embedded in $S^3$ whose regular neighborhood coincides with the handlebody-knot.

\begin{figure}[ht]
        \includegraphics[width=7cm]{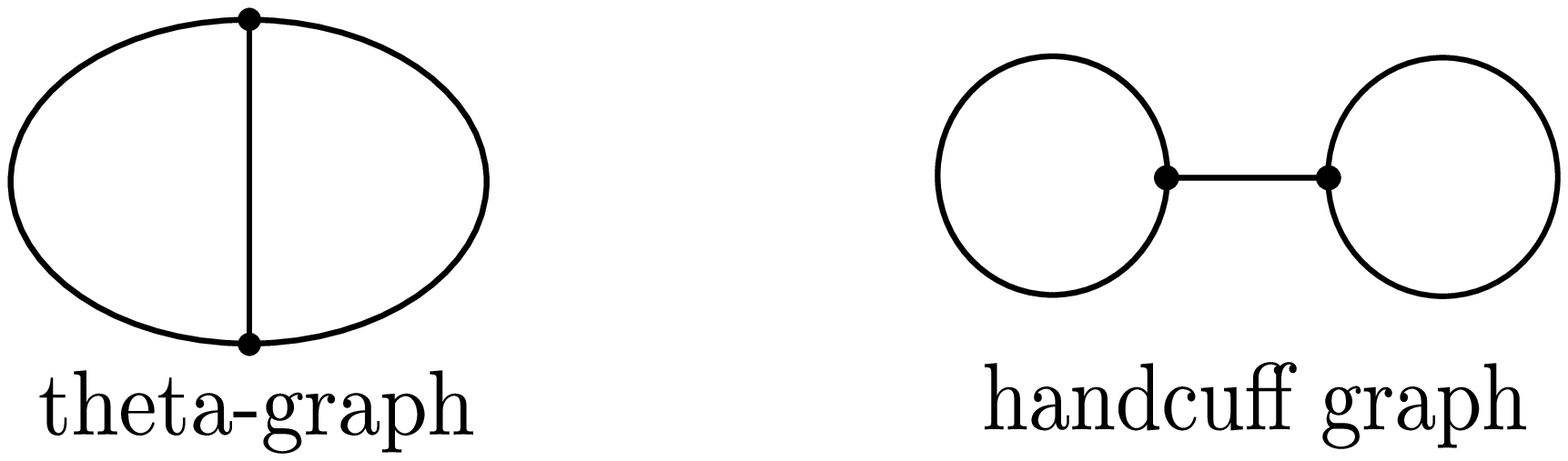}
        \caption{Theta-graph and handcuff graph}
        \label{graphs}
\end{figure}

\begin{defi}
	We say that a handlebody-knot $H$ is \emph{reducible} if there exists a 2-sphere $S^2$ in $S^3$ such that $H \cap S^2$ is an essential 2-disc properly embedded in $H$. Furthermore, we say that a handlebody-knot is \emph{irreducible} if it is not reducible.
\end{defi}

\begin{figure}[ht]
		\includegraphics[width=7cm]{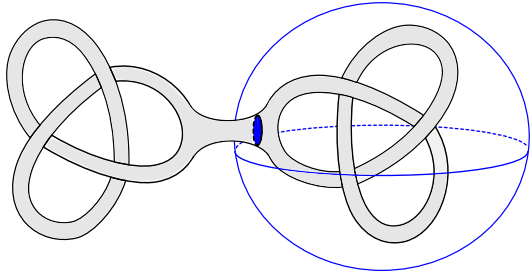}
		\caption{A reducible handlebody-knot}
		\label{red-h.b.knot}
\end{figure}
See Figure~\ref{red-h.b.knot} for an explicit example. The following lemma follows immediately from the definitions .

\begin{lem}\label{lem:irr-prime}
	Let $H$ be a handlebody-knot and let $\partial H$ denote its boundary. Then $H$ is irreducible if and only if $\partial H$ is a prime surface.
\end{lem}

See \cite{Ishii, I-K-M-S} for details of the theory of handlebody-knots, for example. 

\section{Handlebody-knot pairs for surfaces in $S^3$}\label{section:handlebody-knot-pair}

In this section, we describe the relation between the study of surfaces in $S^3$ and the theory of handlebody-knots. The following Fox theorem plays an important role in describing the relation. 

\begin{theo}[Fox \cite{Fox}]\label{theo:Fox}
	Let $M$ be a connected compact $3$-dimensional submanifold of $S^3$ such that its boundary is connected and non-empty. Then there exists a handlebody-knot $H$ such that $S^3\setminus {\rm int}\,H$ is homeomorphic to $M$.
\end{theo}

Let $F$ be a surface in $S^3$ and let $V_F$ and $W_F$ denote the components of the exterior of $F$. By Theorem~\ref{theo:Fox}, there exist handlebody-knots $H_V$ and $H_W$ such that $S^3\setminus {\rm int}\,H_V \approx V_F$ and $S^3\setminus {\rm int}\,H_W \approx W_F$, where the symbol ``$\approx$'' denotes a homeomorphism.

\begin{defi}
	We call the (unordered) pair of the above handlebody-knots $(H_V, H_W)$ a {\it handlebody-knot pair for} $F$.
\end{defi}

\begin{rem}
In general, a handlebody-knot pair for a surface in $S^3$ is not unique. There exist examples of equivalent handlebody-knots with non-homeomorphic exteriors. See \cite{Lee}, for example.
\end{rem}

For a handlebody-knot pair for a surface in $S^3$, we have the following, which is a direct consequence of the definition.
%\begin{sloppypar}
\begin{lem}
	Let $F$ be a surface in $S^3$ and let $(H_1, H_2)$ be a handlebody-knot pair for $F$.
	\begin{itemize}
		\item[$(1)$]If $F$ is unknotted, then both $H_1$ and $H_2$ are trivial handlebody-knots.
		\item[$(2)$]If $F$ is knotted, then $H_1$ or $H_2$ is a non-trivial handlebody-knot.
		\item[$(3)$]If $F$ is bi-knotted, then both $H_1$ and $H_2$ are non-trivial handlebody-knots.
	\end{itemize}
\end{lem}
%\end{sloppypar}

For a prime bi-knotted surface, which is one of the main subjects of our investigation in this paper, we have the following proposition for the case of genus two.

\begin{prop}\label{prop:irr-red}
	Let $F$ be a surface of genus two in $S^3$, and let $(H_1, H_2)$ be a handlebody-knot pair for $F$. If $F$ is a prime bi-knotted surface, then both $H_1$ and $H_2$ are non-trivial handlebody-knots, and one of $H_1$ and $H_2$ is irreducible, and the other is reducible.
\end{prop}

\begin{figure}[ht]
		\includegraphics[width=12cm]{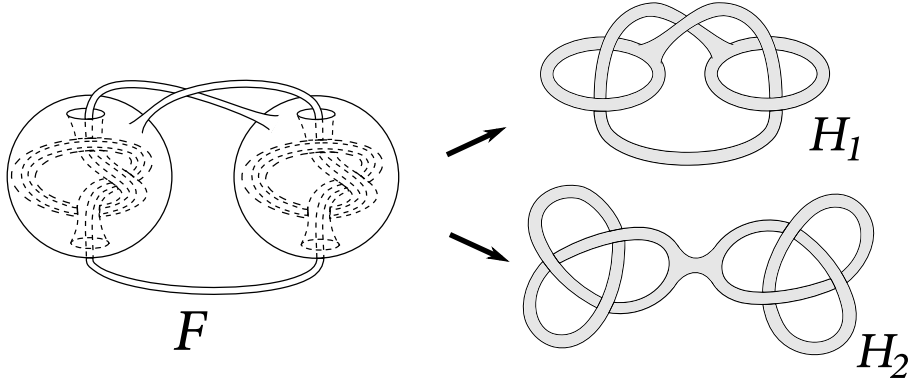}
		\caption{An example of a handlebody-knot pair for a prime bi-knotted surface of genus two}
		\label{h.b.knot-pair}
\end{figure}

The handlebody-knot pair $(H_1,H_2)$ in Figure~\ref{h.b.knot-pair} is an example of a handlebody-knot pair for Homma's example. One can easily observe that $H_1$ is irreducible, while $H_2$ is reducible.
 We will prove Proposition~\ref{prop:irr-red} by using the following two theorems due to Tsukui, Fox and Homma, and their corollaries.

\begin{theo}[Tsukui \cite{Tsukui-2}]\label{theo:Tsukui}
	Let $F$ be a surface of genus two in $S^3$ and let $V_F$ and $W_F$ denote the components of the exterior of $F$. Then $F$ is prime if and only if either the fundamental group $\pi_1 (V_F)$ or $\pi_1 (W_F)$ is indecomposable with respect to free products.
\end{theo}

\begin{cor}\label{cor:irr-indecomp}
	Let $H$ be a genus two handlebody-knot. Then $H$ is irreducible if and only if $\pi_1 (S^3\setminus {\rm int}\, H)$ is indecomposable with respect to free products.
\end{cor}

Corollary~\ref{cor:irr-indecomp} follows directly from Theorem~\ref{theo:Tsukui} and Lemma~\ref{lem:irr-prime}.

A connected compact 3-dimensional manifold $M$ whose boundary is connected and non-empty is said to be \emph{$\partial$-irreducible}, if for every 2-disc $D^2$ properly embedded in $M$, $\partial D^2$ bounds a 2-disc on $\partial M$.

\begin{theo}[Fox \cite{Fox}, Homma \cite{Homma}]\label{theo:Fox-Homma}
	Let $F$ be a surface in $S^3$ and let $V_F$ and $W_F$ denote the components of the exterior of $F$. If $F$ is not equivalent to $S^2$, then at least one of $V_F$ and $W_F$ is not $\partial$-irreducible.
\end{theo}

In fact, for a handlebody-knot $H$ whose genus is greater than or equal to two, $S^3\setminus {\rm int}\, H$ is $\partial$-irreducible if and only if $\pi_1 (S^3\setminus {\rm int}\, H)$ is indecomposable with respect to free products (c.f. \cite{Suzuki}). Thus, the following corollary can be obtained by Theorem~\ref{theo:Fox-Homma} and Corollary~\ref{cor:irr-indecomp}.

\begin{cor}\label{cor:exist-redu}
	Let $F$ be a genus two surface in $S^3$, and let $(H_1,H_2)$ be a handlebody-knot pair for $F$. Then, at least one of $H_1$ and $H_2$ is reducible.
\end{cor} 

\begin{proof}[Proof of Proposition~\ref{prop:irr-red}]
	Since $F$ is bi-knotted, both $H_1$ and $H_2$ are non-trivial. Let $V_F$ and $W_F$ denote the components of the exterior of $F$ such that $V_F \approx S^3\setminus {\rm int}\,H_1$ and $W_F \approx S^3\setminus {\rm int}\,H_2$. Since $F$ is prime, either $\pi_1 (V_F)$ or $\pi_1 (W_F)$ is indecomposable with respect to free products by Theorem~\ref{theo:Tsukui}. We may assume that $\pi_1 (V_F)\,(\cong \pi_1(S^3\setminus {\rm int}\,H_1 ))$ is indecomposable. Then $H_1$ is irreducible by Corollary~\ref{cor:irr-indecomp}. Therefore, by Corollary~\ref{cor:exist-redu}, $H_2$ is reducible. 
\end{proof}

\section{Main theorems}\label{section:main theorem}

So far, we have considered surfaces in $S^3$, especially prime bi-knotted surfaces and their handlebody-knot pairs. In this section, conversely, given two handlebody-knots, we consider the problem whether we can realize them as a handlebody-knot pair for a prime bi-knotted surface in $S^3$, or not. Our main theorem is a partial solution to this problem. We first give some definitions to explain our main theorem. 

\begin{defi}
	Let $G$ be a handcuff graph embedded in $S^3$. Let $v_1,v_2$ be the vertices and $l_1,l_2,\alpha$ be the edges of $G$ such that $l_1,l_2$ are loops, and $\alpha$ is an arc which connects $v_1$ and $v_2$. In particular, we have $l_1 \cup \alpha \cup l_2=G$ and $l_s \cap \alpha =\{v_s\}$, $s=1,2$. Note that $L = l_1 \cup l_2$ is a 2-component link.
	\begin{itemize}
		\item[(1)]We say that a handlebody-knot $H$ satisfies the \emph{property ${\widetilde {\rm T}}$} if $H$ is equivalent to a regular neighborhood of a spatial handcuff graph $G$ such that the associated link $L$ of $G$ is a trivial link.
		\item[(2)]We say that a handlebody-knot $H$ satisfies the \emph{property {\rm T}} if $H$ is equivalent to a regular neighborhood of a spatial handcuff graph $G$ such that at least one of the loops $l_1$ and $l_2$ of $G$ is a trivial knot.
	\end{itemize}
\end{defi} 

\begin{figure}[ht]
		\includegraphics[width=10cm]{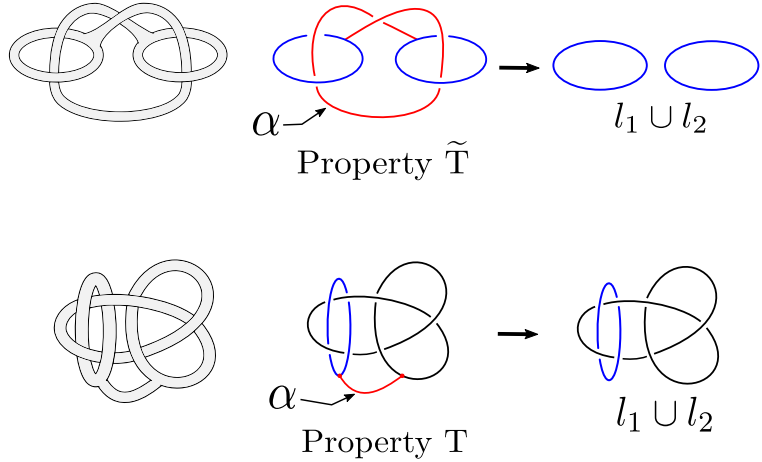}
		\caption{Property ${\widetilde {\rm T}}$ and T}
		\label{property-T}
\end{figure}

See Figure~\ref{property-T} for explicit examples. Note that a handlebody-knot which satisfies the property ${\widetilde {\rm T}}$ satisfies the property T. The following is our main theorem of this paper.

\begin{theo}\label{theo:main}
	$(1)$ Let $H_1$ be a genus two irreducible handlebody-knot which satisfies the property ${\widetilde {\rm T}}$, and $H_2$ be a genus two reducible handlebody-knot which is non-trivial. Then there exists a prime bi-knotted surface $F$ such that $(H_1,H_2)$ is a  handlebody-knot pair for $F$.\\
	\\				
	$(2)$ Let $H_1$ be a genus two irreducible handlebody-knot which satisfies the property {\rm T}, and $H_2$ be a genus two reducible handlebody-knot such that $H_2 = T_1 \natural T_2$, where $T_1$ is a regular neighborhood of a non-trivial knot, $T_2$ is a trivial solid torus, and $T_1 \natural T_2$ denotes the boundary connected sum of $T_1$ and $T_2$ in $S^3$. Then there exists a prime bi-knotted surface $F$ such that $(H_1,H_2)$ is a handlebody-knot pair for $F$.
\end{theo}

\begin{rem}
Recall that for two genus two handlebody-knots which constitute a handlebody-knot pair for a prime bi-knotted surface in $S^3$, it is necessary that both handlebody-knots are non-trivial, and one is irreducible, and the other is reducible by Proposition~\ref{prop:irr-red}.
\end{rem}

\begin{proof}[Proof of Theorem~\ref{theo:main}]
	We will prove part (2) of the theorem. Part (1) can be proved by a similar argument.
	
	Let us first construct a surface $F$ such that $(H_1,H_2)$ is a handlebody-knot pair for $F$. Since a handlebody-knot $H_1$ satisfies the property T, there exists a handcuff graph $G$ embedded in $S^3$ such that $G$ has a trivial knot part, and that regular neighborhood of $G$ is equivalent to $H_1$. Set $G=l_1\cup \alpha\cup l_2$, where $l_1$ and $l_2$ are loops, and $\alpha$ is a simple arc as above. We may assume that $l_1$ is a trivial knot. Then, there exists a 2-disc $D^2$ embedded in $S^3$ such that $\partial D^2 =l_1$. Let $N(D^2)$ be a small regular neighborhood of $D^2$ in $S^3$.  By general position arguments, we may assume that the intersections of ${\rm int}\,D^2$ and $G$ are transverse. Then, we may assume that the intersections of $\partial N(D^2)$ and $G$ are also transverse. So we can set
	 \begin{eqnarray*}
	 	{\rm int}\,D^2 \cap G & = &\{ x_1,x_2,\ldots ,x_k\},\\	 	
	 	\partial N(D^2) \cap G & = & \{\hat{x}_1,\hat{x}_2,\ldots ,\hat{x}_k,\check{x}_1,\check{x}_2,\ldots ,\check{x}_k,p\}
	 \end{eqnarray*}
for some non-negative integer $k$, where $\hat{x}_j$ and $\check{x}_j$ are the end points of the arc component of $N(D^2)\cap G$ containing $x_j$, $j=1, 2, \ldots ,k$, and $p \in \alpha\setminus \{ v_1, v_2\}$ is the end point of the component of $N(D^2)\cap \alpha$ containing the vertex $v_1$ of $G$. We may assume that $\hat{x}_1, \hat{x}_2, \ldots ,\hat{x}_k$ lie in the ``upper side'' of $D^2$, while $\check{x}_1, \check{x}_2, \ldots ,\check{x}_k$ lie in the ``lower side'' of $D^2$. See Figures~\ref{disc-arcs} and \ref{nbd-disc}.
	
	\begin{figure}[ht]
			\includegraphics[width=7cm]{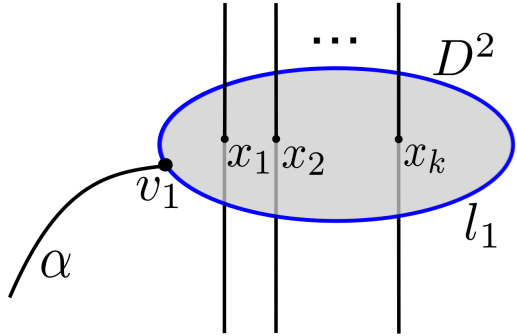}
        	\caption{Trivial knot part of $G$ and $D^2$}
        	\label{disc-arcs}
	\end{figure}
	
	\begin{figure}[ht]
			\includegraphics[width=7cm]{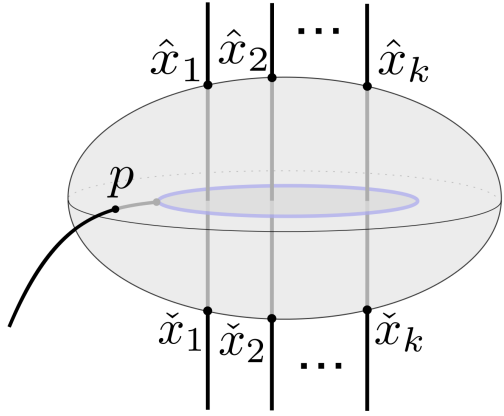}
        	\caption{Intersection of $N(D^2)$ and $G$}
        	\label{nbd-disc}
	\end{figure}   			
	
	Recall that the handlebody-knot $H_2$ is the boundary connected sum of $T_1$ and $T_2$, where $T_1$ is a regular neighborhood of a non-trivial knot and $T_2$ is a trivial solid torus. We denote by $K$ the non-trivial knot in $S^3$ whose regular neighborhood coincides with $T_1$. Let $K_a$ be a 1-tangle in a 3-ball $B^3\subset S^3$ such that for a simple arc $b$ on $\partial B^3$ which connects the two end points of $K_a$, $K_a \cup b$ is equivalent to $K$. Let $h : D^2 \times [-1,1] \to B^3$ be an embedding such that $h(D^2 \times [-1,1])\cap \partial B^3 = h(D^2 \times \{ -1,1 \})$ and $h(\{ 0\} \times [-1,1])=K_a$, where $h(D^2 \times \{1\})$ contains $\hat{x_1},\hat{x}_2, \ldots ,\hat{x}_k$, $h(D^2 \times \{-1\})$ contains $\check{x}_1, \check{x}_2, \ldots ,\check{x}_k$, and $0 \in D^2$ is the center of $D^2$. Set $V(K_a)=h(D^2\times [-1,1])$, which is a regular neighborhood of $K_a$ in $B^3$. Set
	\begin{eqnarray*}		
		T(K_a)=(\partial B^3\setminus V(K_a))\cup (\partial V(K_a)\setminus h({\rm int}\,D^2 \times \{-1,1\})),
	\end{eqnarray*}
which is a torus embedded in $B^3 \subset S^3$. Let $q_1,q_2, \ldots ,q_k$ be distinct points of ${\rm int}\,D^2$. Let $\beta_1,\beta_2,\ldots,\beta_k$ be the simple arcs properly embedded in $V(K_a)$ given by $\beta_j = h(\{ q_j \} \times [-1,1])$ for $j=1,2,\ldots ,k$. See Figure~\ref{surface-arcs} for details, where the figure on the right hand side depicts $T(K_a)$ and $\beta_1,\beta_2,\ldots,\beta_k$ seen from a horizontal direction.

	\begin{figure}[ht]
			\includegraphics[width=9cm]{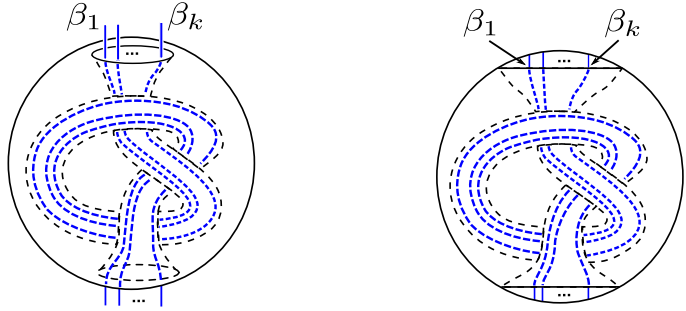}
			\caption{An example of $T(K_a)$ and $\beta_1,\beta_2,\ldots,\beta_k$}
			\label{surface-arcs}
	\end{figure}

	Finally, we replace $N(D^2)$ with $B^3$, i.e. we attach $B^3$ to $S^3\setminus {\rm int}\,N(D^2)$ by using an orientation reversing homeomorphism $f:\partial B^3 \to \partial (S^3\setminus {\rm int}\,N(D^2))$ such that $f(h(q_j,1))=\hat{x}_j$ and $f(h(q_j,-1))=\check{x}_j$, $j=1,2,\ldots ,k$ (see Figure~\ref{changing}). Let $E(K_a)$ be the closure of $B^3\setminus V(K_a)$. Set 
	\begin{eqnarray*}
	G' & = & (G\setminus {\rm int}\,N(D^2))\cup (\beta_1 \cup \beta_2 \cup \cdots \cup \beta_k),\,{\rm and}\\
	W' & = & E(K_a) \cup N(G'), 
	\end{eqnarray*}
where $N(G')$ is a small regular neighborhood of $G'$. Then, the boundary of $W'$ is the required surface $F$. Let $W_F$ denote the component of the exterior of $F$ which contains $\beta_1,\beta_2,\ldots,\beta_k$, and let $V_F$ denote the other component. Then, we see easily that $W_F$ is homeomorphic to $W'$ and to the boundary connected sum of $E(K_a)$ and a solid torus. Hence, $W_F$ is homeomorphic to $S^3\setminus {\rm int}\,H_2$. Furthermore, we also see easily that $S^3\setminus (E(K_a)\cup G')$ is homeomorphic to $S^3\setminus G$. Hence, $V_F$ is homeomorphic to $S^3\setminus {\rm int}\,H_1$.

	\begin{figure}[ht]
			\includegraphics[width=12cm]{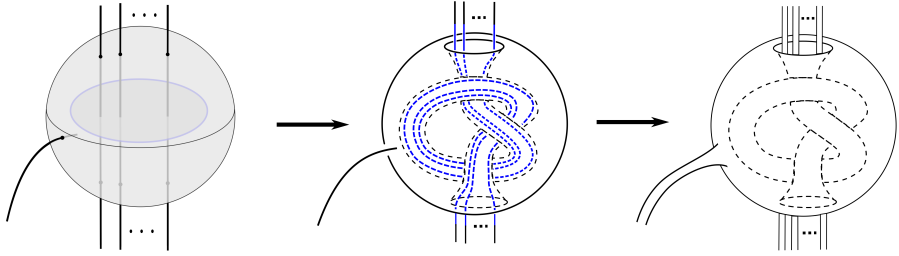}
			\caption{Replacing $N(D^2)$ with $B^3$ and taking the boundary of $W'$}
			\label{changing}
	\end{figure}
	
	Now let us check that the surface $F$ constructed above is prime and bi-knotted. Since $H_1$ is a genus two irreducible handlebody-knot, $\pi_1 (V_F)$ is indecomposable with respect to free products by Corollary~\ref{cor:irr-indecomp}. Hence $F$ is a prime surface by Theorem~\ref{theo:Tsukui}. Furthermore, since neither $\pi_1 (V_F)$ nor $\pi_1 (W_F)$ is a free group, we conclude that $F$ is bi-knotted. This completes the proof of Theorem~\ref{theo:main}.
\end{proof}

	Figures~\ref{main-example-2} and \ref{main-example} give examples of surfaces constructed in the proof of Theorem~\ref{theo:main}. Note that Homma's example (Figure~\ref{homma-sf}) can be obtained by the above construction for the case of (1). 
	
\begin{rem}
	It is probable that our construction leads to an infinite family of surfaces that have $(H_1, H_2)$ as a handlebody-knot pair. Note that the same construction method works even if we replace the homeomorphism $h : D^2 \times [-1,1] \to V(K_a)$. For every integer $n \in \mathbb{Z}$, let $h_n : D^2 \times [-1,1] \to V(K_a)$ be the homeomorphism obtained from $h$ by applying the $n$ times full twist along $D^2 \times \{0\}$. It yields an infinite family of surfaces $\{F_n\}$ that have $(H_1, H_2)$ as a handlebody-knot pair, although we do not know if $F_i$ is not equivalent to $F_j$ for $i \neq j$.
\end{rem}

\begin{figure}[ht]
		\includegraphics[width=10cm]{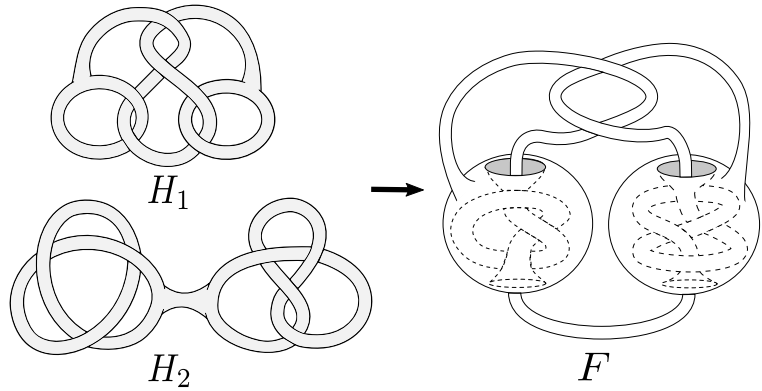}
		\caption{Realizing handlebody-knots $H_1$ and $H_2$ as a handlebody-knot pair for a prime bi-knotted surface $F$ (Theorem~\ref{theo:main}\,(1))}
		\label{main-example-2}
\end{figure}

\begin{figure}[ht]
		\includegraphics[width=10cm]{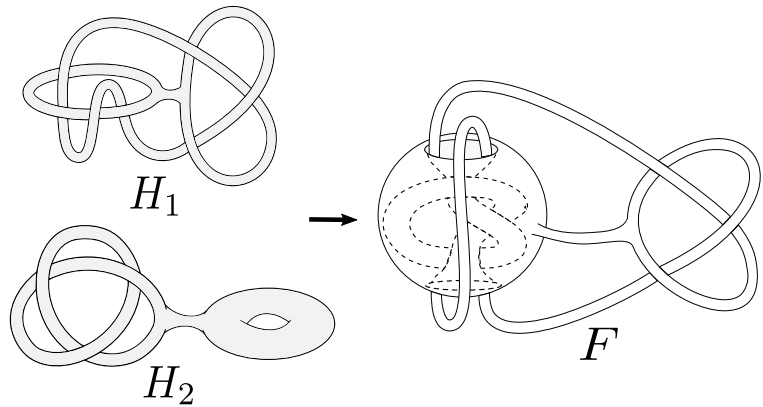}
		\caption{Realizing handlebody-knots $H_1$ and $H_2$ as a handlebody-knot pair for a prime bi-knotted surface $F$ (Theorem~\ref{theo:main}\,(2))}
		\label{main-example}
\end{figure}

\bibliographystyle{amsrefs}

\begin{thebibliography}{99}
\bibitem{Alexander}
J.W. Alexander, ``On the subdivision of 3-space by a polyhedron", Proc. Natl. Acad. Sci. USA {\bf 10} (1924), 6--8.

\bibitem{Fox}
R.H. Fox, ``On the imbedding of polyhedra in 3-space'', Ann. of Math. {\bf 49} (1948), 462--470.

\bibitem{Homma}
T. Homma, ``On the existence of unknotted polygons on 2-manifolds in $\mathbb{E}^3$", Osaka Math. J. {\bf 6} (1954), 129--134.

\bibitem{Ishii}
A. Ishii, ``Moves and invariants for knotted handlebodies''，Algebr. Geom. Topol. {\bf 8} (2008), 1403--1408.

\bibitem{I-K-M-S}
A. Ishii, K. Kishimoto, H. Moriuchi and M. Suzuki, ``A table of genus two handlebody-knots up to six crossings'', J. Knot Theory Ramifications {\bf 21} (2012), 1250035.

\bibitem{Lee}
J.H. Lee and S. Lee, ``Inequivalent handlebody-knots with homeomorphic complements", Algebr. Geom. Topol. {\bf 12} (2012), 1059--1079.

\bibitem{Suzuki}
S. Suzuki, ``On surfaces in 3-sphere: prime decompositions", Hokkaido Math. J. {\bf 4} (1975), 179--195.

\bibitem{Tsukui-1}
Y. Tsukui, ``On surfaces in 3-space", Yokohama Math. J. {\bf 18} (1970), 93--104.

\bibitem{Tsukui-2}
Y. Tsukui, ``On a prime surface of genus 2 and homeomorphic splitting of 3-sphere", Yokohama Math. J. {\bf 23} (1975), 63--75.

\bibitem{Waldhausen}
F. Waldhausen, ``Heegaard-Zerlegungen der 3-Sph\"{a}re'', Topology {\bf 7} (1968), 195--203.

\end{thebibliography}

\end{document}